%% file: main.tex
\documentclass[11pt]{article}

\usepackage{amssymb}
\usepackage{amsmath}
\usepackage{amsthm}
\usepackage{amsfonts,amsthm,amsmath,amssymb}

\usepackage{dsfont}

\usepackage{hyperref}
\hypersetup{colorlinks,allcolors=black}

\usepackage{thmtools}
\usepackage{thm-restate}
\usepackage[noabbrev]{cleveref}
\newtheorem{thm}{Theorem}[section]
\newtheorem{theorem}[thm]{Theorem}
\newtheorem{lemma}[thm]{Lemma}
\newtheorem{corollary}[thm]{Corollary}
\newtheorem{claim}[thm]{Claim}

\newtheorem{fact}[thm]{Fact}

\newtheorem{remark}[fact]{Remark}

\newtheorem{definition}[thm]{Definition}
\newtheorem{construction}[thm]{Construction}
\newtheorem{example}[thm]{Example}

\usepackage{mathtools}

\newcommand\bb[1]{{\mathbf{#1}}}

\usepackage[ruled,vlined]{algorithm2e}

\newcommand\lat{ {\mathcal L} }

\newcommand\B{ {\mathcal B} }

\newcommand\x{ {{\vec x}} }

\newcommand\R{\mathbb{R}}
\newcommand\Z{\mathbb{Z}}
\newcommand\N{\mathbb{N}}
\newcommand\Q{\mathbb{Q}}

\newcommand\F{\mathbb{F}}

\newcommand\set[1]{{\left\{ #1 \right\}}}

\newcommand\sett[2]{\left\{ #1 \;\middle\vert\; #2 \right\}}

\newcommand\norm[1]{\left\| #1 \right\|}

\newcommand\half{{\frac{1}{2}}}

\newcommand\defeq{\stackrel{def}{=}}

\newcommand\eps{\varepsilon}

\newcommand\spn{span}

\newcommand\parentheses[1]{{\left({#1}\right)}}
\newcommand\ps[1]{\parentheses{#1}}

\newcommand\abs[1]{{\left\lvert{#1}\right\rvert}}

\usepackage{soul}
\usepackage{xcolor}

\usepackage{wasysym}
\usepackage{comment}
\usepackage{marvosym}
\usepackage{tikz}
\usetikzlibrary{calc}
\usepackage[cm]{fullpage}
\usepackage{setspace}
\usepackage[english]{babel}
\usepackage[square,numbers]{natbib}
\usepackage{tikz}  
\usepackage{tikz-3dplot} 
\usepackage{amssymb}
\usepackage{xifthen}
\usepackage{caption}
\usepackage{subcaption}
\renewcommand{\vec}[1]{\ensuremath{\boldsymbol{#1}}}
\renewcommand{\spn}{\mathrm{Span}}

\title{On the Shortest Lattice Vector vs. the Shortest Basis}
\author{Yael Eisenberg\thanks{Cornell University. This material is based upon work supported by the National Science Foundation Graduate
Research Fellowship under Grant No. DGE - 2139899. email: \url{ye45@cornell.edu}} \and Itamar Rot\thanks{School of Computer Science, Tel Aviv University. Supported by the European Research Council (ERC) under the European Unions Horizon 2020 research and innovation programme (Grant agreement No. 835152). email: \url{itamarrot@mail.tau.ac.il}} \and Muli Safra\thanks{School of Computer Science, Tel Aviv University. Supported by the European Research Council (ERC) under the European Unions Horizon 2020 research and innovation programme (Grant agreement No. 835152) and by the Israel Science Foundation (ISF) grant 2257/21. email: \url{safra@mail.tau.ac.il}}}
\date{}

\begin{document}
\setstretch{1.5}

\maketitle
 
\begin{abstract}
Given an arbitrary basis for a mathematical lattice, to find a ``good" basis for it is one of the classic and important algorithmic problems.
In this note, we give a new and simpler proof of a theorem by \cite{regavim2021minkowski}:
we construct a 18-dimensional  lattice that does not have a basis that satisfies the following two properties simultaneously:
\begin{enumerate}
    \item The basis includes the shortest non-zero lattice vector.
    \item The basis is shortest, that is, minimizes the longest basis vector 
    (alternatively: the sum or the sum-of-squares of the basis vectors).
\end{enumerate}
The vectors' length can be measured in any $\ell^q$ norm, for $q\in \N_+$ (albeit, via another lattice, of a somewhat larger dimension).
\end{abstract}

\newpage
\section{Introduction}

\input{intro.tex}

\section{Preliminaries}

The $\ell^q$ norm of a vector $\vec{x}\in \R^n$ is denoted as $\norm{\vec{x}}_q=\ps{\sum_{i=1}^n \abs{x_i}^q}^\frac{1}{q}$.
The centralized $\ell^q$ ball of radius $r$ is denoted $\B_q(r) \defeq \set{x \in \R^n \mid \norm{\vec{x}}_q < r}$.

We notate the Minkowski sum of sets $\vec{A}+\vec{B}=\set{\vec{x}+\vec{y}\mid \vec{x}\in \vec{A},\vec{y}\in \vec{B}}$ for $\vec{A},\vec{B} \subset \R^n$.

\subsection{$\ell^q$ Norm Modulo}

The grid scaled by a real number  $m$ is denoted $m\Z^n$;
one may measure the distance of a vector from that set of points:
the absolute value of $\alpha \mod m$ is $\abs{\alpha}_m\defeq\min_{z\in \Z} \abs{\alpha-z}$. 

Observe that $\abs{\alpha}_m=\abs{-\alpha}_m$ and $\abs{\alpha}_m = k$ if and only if $\alpha\mod m = \pm k$.

\begin{definition}[$\ell^q$-norm modulo $m$]
The $\ell^q \mod m$ norm of $\vec{x} \in \R^n$ is \[\norm{\vec{x} \mod m}_q=\ps{\sum_{i=1}^n \abs{x_i}_m^q}^\frac{1}{q} \]
namely, it is the $\ell^q$---distance of $\vec{x}$ from $m\Z^n$. 
\end{definition}

Observe that $\norm{\vec{x}}_q \geq \norm{\vec{x} \mod m}_q$, and there is an equality if and only if $\vec{x} \in {\left[-\frac{m-1}{2},\frac{m-1}{2}\right]}^n$.

\subsection{Lattices}

\begin{definition}
A lattice  $\lat$ is a discrete subgroup of $\R^n$.

Equivalently, a matrix $M \in \R^{n\times k}$ with linearly independent columns specifies a lattice $\lat = M\Z^k$. 
Such a matrix is referred to as \emph{a basis of the lattice} and the lattice is then denoted as $\lat[M]$.
$n$ is the \emph{dimension} of the lattice and $k$ is the \emph{rank} of the lattice. A lattice is called \emph{full rank} if $n=k$. We will treat lattices as though they are full rank by implicitly identifying $\spn(\lat)$ with $\R^k$.
\end{definition}
%\todo{if we choose to treat all lattices as tho they are full rank, we can choose $n$ or $k$ and stick with it}

Any lattice of rank greater than $1$ can be generated by infinitely many bases. The following fact specifies exactly when two bases span the same lattice:
\begin{fact}
$M,M'$ are bases of the same lattice, $\lat[M]=\lat[M']$, if and only if there exists $\vec{U} \in \mathrm{GL}_k(\Z)$ such that $M'=MU$.
\end{fact}

As mentioned above, there are several possible manners to measure the length of a lattice basis. 
In this paper, we use the following definition:
\begin{definition}[Shortest lattice basis]
Let $\lat\subset \R^n$ be a lattice of rank $k\leq n$.
The \emph{length of $\lat$'s shortest basis}, denoted $\overline{\lambda}(\lat)$, is the basis of $\lat$ which minimizes the longest basis vector:
\[\overline{\lambda}^{(q)}(\lat) \defeq \min_{\vec{v}_1,\ldots \vec{v}_k \in \R^{n}\colon 
\lat[\vec{v}_1,\ldots,\vec{v}_k]=\lat} \, \, \max_{1\le i \le k} \norm{\vec{v}_i}_q \]
and any basis of $\lat$ that achieves the minimum length is \emph{a shortest basis of $\lat$}.
\end{definition}

There are other possible manners by which to measure the length of a basis:
\begin{definition}
The \emph{$\ell^{q'}$ length} of a matrix $M = \left[ \vec{v}_1, \ldots, \vec{v}_k \right] \in \R^{n\times k}$ is 
\[\norm{M}_q^{q'} = \ps{\sum_{i=1}^k \norm{\vec{v}_i}_q^{q'}}^\frac{1}{q'} \]
for any $1\leq q' < \infty$.
\end{definition}

Let us establish a simple lower bound on the length of any basis for a particular lattice: 

\begin{definition} \label{definition: lambdas}
The \emph{successive minima} of a lattice are
\[ \lambda_i^{(q)}(\lat) \defeq\inf \set{ r>0 \mid \dim \spn{(\lat \cap \B_q(r))} \geq i} \]
\end{definition}

Observe that $\lambda_1^{(q)}(\lat)$ is the length of the shortest non-zero lattice vector.
Further, observe that a basis of $\lat$ that achieves the successive minima is a shortest basis of $\lat$; this is obtainable provided the lattice is standard as elaborated next.

\begin{fact}[Folklore] \label{fact:linearind}
    For any $q\in\N$, lattice $\lat \subset \R^n$ of rank $k$ there exist linear independent vectors $\vec{u}_1,\ldots,\vec{u}_k \in \lat$ such that $\norm{\vec{u}_i}_q=\lambda_i^{(q)}(\lat)$.
\end{fact}

\section{Our characterization of non-standard lattices} \label{sec:characterization}

Once the successive minima are introduced, 
a natural question to ask is whether every lattice has a basis such that the basis vectors achieve the successive minima. 
The answer to this question is a resounding `No':
\begin{example} \label{example:non-standard}\cite{van1956reduktionstheorie}
Let $n \geq 2^q+1$ and consider the $n$-dimensional lattice basis
\[M \defeq 
\begin{pmatrix}1 &  &  & \frac{1}{2}\\
 & \ddots &  & \vdots\\
 &  & 1 & \vdots\\
0 & \cdots & 0 & \frac{1}{2}

\end{pmatrix}\]
% \[ \lat \defeq \lat\left[\begin{pmatrix}1 &  &  & \frac{1}{2}\\
%  & \ddots &  & \vdots\\
%  &  & 1 & \vdots\\
% 0 & \cdots & 0 & \frac{1}{2}
% \end{pmatrix} \right] \subset \R^n \]
Observe that $\pm \vec{e}_1,\ldots,\pm \vec{e}_n$ are the shortest vectors of the lattice $\lat[M]$ and thus $\lambda_1^{(q)}=\ldots=\lambda_n^{(q)} = 1$.
Any basis of $\lat$, however, must include a vector $\vec{v}\in\frac{1}{2}\Z_{odd}^n$, where $\norm{\vec{v}}\geq\half n^\frac{1}{q}>1$. 
\end{example}

\begin{definition}
A lattice $\lat \subset \R^n$ of rank $k$ is \emph{$\ell^q$-standard} if there exists a basis $M=[\vec{v}_1,\ldots,\vec{v}_k]$ of $\lat$ such that for any $1\le i \le k$, $\norm{\vec{v}_i}_q = \lambda_i^{(q)}(\lat)$. 
% Equivalently: $\lambda_n(\lat)=\overline{\lambda}(\lat)$ (yael: not convinced of this last part. lets say we have an orthogonal lattice, generated $(1,0,0)$, $(0,1,0)$ and $(0,0,10000)$. This is the standard basis. however, a non-standard basis would be $(1,0,0)$, $(1,1,0)$ and $(0,0,10000)$ which also fulfills $\lambda_n(\lat)=\overline{\lambda}(\lat)$, since it minimizes the largest vector in the basis.
\end{definition}

Observe that in a standard lattice, $\overline{\lambda}^{(q)}(\lat) = \lambda_n^{(q)}(\lat)$. 
In contrast, the lattice of \Cref{example:non-standard} is not standard, and in fact
\[ \overline{\lambda}^{(q)}(\lat)=\half n^\frac{1}{q}>1= \lambda_n^{(q)}(\lat).\]

We remark that any lattice of dimension $4$ or lower must be $\ell^2$-standard \cite{feng2017standard}. An extension of this result is discussed in \Cref{section:Discussion}. 
Moreover, it is clear that the lattice in \Cref{example:non-standard} is constructed by beginning with $\Z^n$ and adding a new vector, which is a vector of the grid divided by some natural number.
In \Cref{theorem: Characterization of non-standard lattices} we demonstrate the fact that any non-standard lattice is constructed in this manner.
\begin{theorem}[Characterization of non-standard lattices]
\label{theorem: Characterization of non-standard lattices}
Let $\lat\subset\R^k$ be any non-standard full rank lattice, let $q\in\N$; let $W_1,\dots,W_k\in\lat$ be a collection of linearly independent vectors with $\norm{W_i}^{q}=\lambda_i^{(q)}(\lat)$ (those exist by \Cref{fact:linearind}) and let $\lat_{\mathrm{std}}\subset\lat$ be the lattice generated by the $W_i$s, i.e., $\lat_{\mathrm{std}}\defeq\lat[W]$. Then there are vectors $\vec{u}_1,\ldots,\vec{u}_\ell \in \lat_{\mathrm{std}}$ and integers $n_1,\ldots,n_\ell \in\N$ ($\ell \leq k$)
such that 
\[ \lat = \spn_\Z{\ps{{\lat_{\mathrm{std}}} + \set{\frac{\vec{u}_1}{n_1},\ldots,\frac{\vec{u}_\ell}{n_\ell}}}}.\]

%Then, there are vectors $\vec{v}_1,\ldots,\vec{v}_k \in \lat_{\mathrm{std}}$ and integers $n_1,\ldots,n_k \geq 2$
%such that
%\[ \lat = \lat_{\mathrm{std}} + \spn_\Z{\ps{\set{\frac{\vec{v}_1}{n_1},\ldots,\frac{\vec{v}_k}{n_k}}}}.\]

%then there exists a standard sub-lattice $\lat_{\mathrm{std}}\subseteq\lat$, so that 
%\begin{itemize}
   % \item  $\forall i \le n\colon\lambda_i^{(q)}(\lat_{\mathrm{std}})=\lambda_i^{(q)}(\lat)$ 
   % \item there are vectors $\vec{u}_1,\ldots,\vec{u}_k \in \lat_{\mathrm{std}}$ and integers $n_1,\ldots,n_k \geq 2$
%such that
%\[ \lat = \lat_{\mathrm{std}} + \spn_\Z{\ps{\set{\frac{\vec{u}_1}{n_1},\ldots,\frac{\vec{u}_k}{n_k}}}}.\]
%\end{itemize}    
\end{theorem}
%\todo{do we need to add the $\lat_{std}$?}

\begin{proof}
%Clearly, there exists $w_l\in \lat$ such that $\norm{w_l} = \lambda_l$ and $w_1,\ldots,w_l$ are linearly independent (via induction on $l$). 
Let $W_1,\dots,W_n\in\lat$ be a collection of linearly independent vectors with $\norm{W_i}^{q}=\lambda_i^{(q)}(\lat)$ (those exist by \Cref{fact:linearind}).

Denote $W\defeq[W_1,\ldots,W_n]$ and further denote $\lat_{\mathrm{std}}\defeq\lat[W]$. 
    
The successive minima of a sub-lattice cannot decrease moving from $\lat$ to $\lat_{\mathrm{std}}$, as $\lat_{\mathrm{std}} \subseteq\lat$ and thus $\forall i\in[n]\colon\lambda_i^{(q)}(\lat_{\mathrm{std}})\geq \lambda_i^{(q)}(\lat)$.
On the other hand, $W_1,\ldots, W_n \in \lat_{\mathrm{std}}$ and thus $\lambda_i^{(q)}[\lat_{\mathrm{std}}] \leq \norm{W_i}_q =\lambda_i^{(q)}[\lat]$;
altogether it must be the case that $\lat$ and $\lat_{\mathrm{std}}$ have the same successive minima, $\forall i\in[n]\colon\lambda_i^{(q)}(\lat_{\mathrm{std}})=\lambda_i^{(q)}(\lat)$. 
    
Let $M=\left[\vec{v}_1,\ldots,\vec{v}_k\right]$ be a basis of the original lattice $\lat$. 
Since $W_1,\ldots,W_k \in \lat[M]$, there exist $C_1,\ldots,C_k \in \Z^n$ such that $W_i = MC_i$. 
Denote $C=[C_1,\ldots,C_k] \in \Z^{k \times k}$ and observe that $W=MC$. 
Since $W,M$ are invertible,
%\todo{make sure $M$ is full rank, or multiply by transpose of orthogonal matrix to force it to be full rank / square matrix}
so is $C$, and hence $WC^{-1}=M$. 
$C$ is an integer matrix and hence $C^{-1}\in\Q^{n\times n}$. 
% If $C^{-1}_i \in \Z^{n\times n}$ we skip it,  otherwise:  
We denote $C^{-1}_{ij} = \frac{r_{ij}}{t_{ij}}$, $n_i = \mathrm{lcm}(t_{i1},\ldots,t_{in})$ and $n_i\vec{v}_i = \sum_{k=1}^n \underbrace{\frac{r_{ik} \cdot n_i}{t_{ik}}}_{\in\Z} W_k \in \lat[W] $. It immediately follows that $\vec{v}_i = \frac{1}{n_i}\vec{u}_i$ for a particular $\vec{u}_i\in\lat[W]$ and the theorem is proven.

%It immediately follows that $\vec{v}_i = \frac{r_{i1}}{t_{i1}} W_1 + \ldots + \frac{r_{in}}{t_{in}}W_n = \frac{1}{n_i}\vec{u}_i$ for some $\vec{u}_i\in\lat[W]$ and the theorem is proven.

%\[ M_i = WC^{-1}_i = \frac{r_{i1}}{t_{i1}} W_1 + \ldots + \frac{r_{in}}{t_{in}}W_n = \frac{1}{lcm(t_{i1},\ldots,t_{in})}\underbrace{\ps{\sum_{k=1}^n \underbrace{\frac{r_{ik} \cdot lcm(t_{i1},\ldots,t_{in})}{t_{ik}}}_{\in \Z} W_k}}_{\in \lat[W]}. \]

%Denote by $n_i = lcm(t_{i1},\ldots,t_{in})$ and $\vec{v}_i = \sum_{k=1}^n \frac{r_{ik} \cdot lcm(t_{i1},\ldots,t_{in})}{t_{ik}} W_i \in \lat[W] $.

%$\lat_{\mathrm{std}} , \spn_\Z{\ps{\set{\frac{\vec{v}_1}{n_1},\ldots,\frac{\vec{v}_k}{n_k}}}}$
%are both contained in $\lat$ and $M_1,\ldots M_n$ are in
%$\lat_{\mathrm{std}} + \spn_\Z{\ps{\set{\frac{\vec{v}_1}{n_1},\ldots,\frac{\vec{v}_k}{n_k}}}}$.
%Therefore $\lat = \spn_{\Z}(M_1,\ldots,M_n) \subset \lat_{\mathrm{std}} + \spn_\Z{\ps{\set{\frac{\vec{v}_1}{n_1},\ldots,\frac{\vec{v}_k}{n_k}}}}$
\end{proof}

\begin{remark}
For the $q=2$ case, one can also require the length of the vectors $\frac{\vec{u}_1}{n_1},\ldots,\frac{\vec{u}_k}{n_k}$ to be at most $\frac{\sqrt{n}}{2}\lambda_n(\lat)$
(using the fact that the distance of any vector in $\R^n$ from $\lat_{\mathrm{std}}$ is at most $\frac{\sqrt{n}}{2}\lambda_n(\lat_\mathrm{std})=\frac{\sqrt{n}}{2}\lambda_n(\lat)$ \cite{micciancio2002complexity}, and that shifting $\frac{\vec{u}_i}{n_i}$ by a vector in $\lat_{\mathrm{std}}$ does not change $ \lat_{\mathrm{std}} + \spn_\Z{\ps{\set{\frac{\vec{u}_1}{n_1},\ldots,\frac{\vec{u}_k}{n_k}}}}$).
\end{remark}

\section{A Construction and its proof} \label{sec:construction}

\subsection{A Man, a Plan, a Canal} \label{subsec:plan}

This section is devoted to a proof of \Cref{thm:Short basis vs. shortest vector} which
demonstrates a lattice for which any short basis cannot include the shortest non-zero lattice vector.

\begin{theorem}[Shortest basis vs. shortest vector]\label{thm:Short basis vs. shortest vector}
For any integer $q\in\N$ and  large enough $n$,
there exists a non-standard lattice $\lat \subset \R^n$, for which, any basis $M=\left[ \vec{v}_1,\ldots,\vec{v}_n\right]$ with $\max_{1\le i \le n} \norm{\vec{v}_i}_q = \overline{\lambda}^{(q)}(\lat)$, must have $\lambda_1^{(q)}(\lat) < \min_{1\le i \le n} \norm{\vec{v}_i}_q $---that is, any shortest basis does not achieve $\lambda_1^{(q)}(\lat)$.
The dimension $n$ of $\lat$ could be as low as 18 for $q=2$.

% \todo{shouldn't this be for any dimension $n\geq 5$?} 
\end{theorem}
The lowest dimension example we found fulfilling \Cref{thm:Short basis vs. shortest vector} is in \Cref{thm:main}.
\noindent
The construction proceeds in two stages:
\begin{enumerate}
\item First, construct a non-standard lattice $\lat_+ \subset \R^n$ of rank $n$, with exactly $2(n+1)$ shortest vectors (of equal length): 
add a vector $\vec{u}$ of norm $p\in\R$ to the scaled integer grid $p\Z^n$.
%by, taking the integer grid $\Z^n$, stretch it by a factor $p$, then extend it into a non-standard lattice by adding another vector $u$ of norm $p$ to the lattice.
The only vectors of norm $p$ in $\lat_+$ (up to sign) are $W_1=p\cdot \vec{e}_1,\ldots,W_n=p\cdot \vec{e}_n,\vec{u}$. %\todo{lets choose capital or lowercase for $W_i$ and be consistent}
\item Construct a new lattice $\tilde{\lat}$, by adding a coordinate to the $W_i$ vectors to create $\tilde{W}_i$ vectors, such that $\norm{\tilde{W}_1},\ldots,\norm{\tilde{W}_n},\norm{\tilde{\vec{u}}} \approx p$.
The additional coordinate will be added in a manner that forces $\tilde{\vec w}_1$ to be the only shortest lattice vector. We will prove that the shortest basis of $\tilde{\lat}$ does not include $\tilde{\vec w}_1$.
\end{enumerate}

\include{draw.tex}

It is still not clear how to ensure that no short basis of $\tilde{\lat}$ includes the shortest vector $\vec{\tilde{W}}_1$. 
As in \Cref{theorem: Characterization of non-standard lattices}, we can write
\[\vec{u} = \frac{r_1}{t_1}W_1 + \ldots + \frac{r_n}{t_n}W_n \]
where $r_1,\ldots,r_n,t_1,\ldots,t_n \in \Z$.
Since the shortest vectors in $\lat_+$, as well as in $\tilde{\lat}$, are $W_1,\ldots,W_n,\vec{u}$,
the shortest basis may only include those.
If we construct $u$ such that $r_i \neq \pm 1$, the basis will have to include $W_i$.
So, we must select the coefficients of linear combination accordingly. 

Once establishing this requirement, it is not clear how to find the vector $\vec{u}$, which should have the same norm as $W_1,\ldots,W_n$, nevertheless, so that adding $\vec{u}$ to the lattice $\lat_+$ would not create any vectors with smaller norms.

Indeed, we may consider the following attempt (for the Euclidean norm): take the scaled grid $p\Z^n$ and add the vector $\vec{u}=(2,\ldots,2,1)$.
Since $\vec{u}$ should be a vector on the unit sphere, we must choose $p=\sqrt{4n-3}$.
However, $p\Z^n + \vec{u}$ includes the vector $\frac{p+1}{2}\vec{u}-p\cdot \vec{e}_1-\ldots-p\cdot\vec{e}_n=(1,\ldots,1,\frac{-p+1}{2})$ 
which has length $\sqrt{n-1+\frac{(p-1)^2}{4}} \approx \frac{p}{\sqrt{2}} < p$---that is, we have accidentally added a vector shorter than $p$ to the lattice.

To solve this problem, we use a vector with $2$'s and $3$'s entries.
Those have the property that if $\abs{2r}_p$ decreases, then $\abs{3r}_p$ must increase.

\subsection{Stage I: amending the Grid}  \label{subsec:stage1}

% \begin{lemma}
% There exists a prime $p\in \N$ and a vector $x \in (\Z_p \setminus \set{0,\pm1})^n$ such that  $\norm{x}^2=p^2-1$ such that for any $0,\pm 1\neq k \in \Z_p$, $\norm{x}^{(p)}<\norm{kx}^{(p)}$.
% \end{lemma}

% \begin{proof}
% (I had written a script and found  it :))
% \[ x = (5, 3, 3, 7, 2, 2, 5, 3, 4, 4, 6, 5, 6, 5, 1), p = 17
% \]
% \end{proof}

%here is a more abstract way to prove such an existence

\begin{claim}
\label{thm:A short vector modulo prime}[A Short Vector]
For any large enough prime $p$, there exists $n\in \N$ and  $\mbox{$\vec{\sigma}$} \in {\left[-\frac{p-1}{2},\frac{p-1}{2}\right]}^{n-1}$ such that $\norm{\mbox{\boldmath $\sigma$} \mod p}_q = \ps{p^q-1}^\frac{1}{q}$, and for any $k \in \Z_p \setminus \set{0,1,-1}$, $\norm{k\mbox{\boldmath $\sigma$} \mod p}_q > \norm{\mbox{\boldmath $\sigma$} \mod p}_q$.    
\end{claim}
The proof that such a vector exists appears in~\Cref{sec:A short vector modulo prime}.

Observe that the vector $(1,\mbox{\boldmath $\sigma$})$ is still the shortest up to multiplication in $\F_p$.

\begin{construction} \label{construction: n shorts}
Let $p,n$ and $\mbox{\boldmath $\sigma$} \in {\left[-\frac{p-1}{2},\frac{p-1}{2}\right]}^{n-1}$ be defined as in \Cref{thm:A short vector modulo prime}, and denote $\vec{u} = (1,\mbox{\boldmath $\sigma$})$. 
Let $W_i = p\vec{e}_i \in \R^n$ be the $p$-streched $n$-unit vectors.
Define the $n$-dimensional lattice
\[ \lat_+ = \lat[W_1,\ldots,W_n,\vec{u}] \]
\end{construction}

\begin{lemma}
$\pm W_1,\ldots,\pm W_n,\pm \vec{u}$ are the only vectors in $\lat_+$ with $\ell^q$ norm $p$, and they are the shortest vectors in the lattice (i.e., there are no non-zero lattice vectors in $\lat_+$ with norm $< p$).
\end{lemma}

\begin{proof}
First, observe that the norm of $\vec{u}$ is $\ps{\sum_{i=1}^{n-1} \sigma_i^q + 1}^\frac{1}{q}=p$. 

Assume, by way of contradiction, that there is another short lattice vector $\vec{x}\in\lat_+$, that is, $0,\pm W_1,\ldots,\pm W_n, \pm \vec{u} \neq \vec{x}$ such that
$\norm{\vec{x}}_q \le p$. 
Clearly, $\vec{x} \notin \lat[W_1,\ldots,W_n] = p \Z^n$, and hence there is $k \in \Z \setminus p\Z$ and $\vec{z} \in \Z^n$ such that $\vec{x} = k\vec{u} + p\vec{z}$. 
We may also assume, without loss of generality, that $0 < k \le p-1$ (if $k=pt+r$ where $t\in \Z$ and $r<p$, then $\vec{x}=r\vec{u} + p(t\vec{u}+\vec{z})$).
Now, observe that
\[ \norm{\vec{x}}_q \geq \norm{\vec{x} \mod p}_q = \norm{k\vec{u}+p\vec{z} \mod p}_q =  \norm{k\vec{u} \mod p}_q \geq \norm{\vec{u} \mod p}_q  = \norm{\vec{u}}_q = p; \]
here we use the fact that the norm of vectors in $[-\frac{p-1}{2} , \frac{p-1}{2}]^{n}$ does not change modulo $p$. 
By \Cref{thm:A short vector modulo prime}, if $k\neq \pm 1$, then $\norm{\vec{x}}>p$, which is a contradiction;
therefore, consider the case that $k = \pm1$ and, without loss of generality, the case $k=1$. 
In this case, $\vec{x}=\vec{u} + p\vec{z}$, and such a vector is the shortest if it is in ${\left[-\frac{p-1}{2},\frac{p-1}{2}\right]}^n$. This is the case if and only if $\vec{z}=\vec{0}$, and thus $\pm \vec{u}$ is the shortest vector in $\lat_+ \setminus p\Z^n$. 
\end{proof}

\subsection{Stage II: shaking the lattice}  \label{subsec:fixing}

After constructing $\lat_+ \subset \R^n$, a lattice with $2(n+1)$ vectors of the shortest length, we are able to proceed to construct our ``holy grail" lattice. 
We start with the lattice $\lat_+$, and  add a tiny noise (without increasing the rank of the lattice), such that $\vec{v}_1=(W_1,\eps)$ will be the shortest lattice vector.

\begin{construction}
Let $\lat_+ \subset \R^n$ be as in \Cref{construction: n shorts}. Since a lattice is a discrete group, there exists $R>p$ such that $\lat_+ \cap \B_q(R) = \set{\vec{0},\pm W_1,\ldots,\pm W_n,\pm \vec{u}}$. 
Now, let $0<\eps < \frac{p^q(R^q-p^q)}{((n-1)(p-1)+1)^q},\frac{R-p}{2}$
and define some new vectors:

\[ \vec{v}_1 = (W_1, \eps),\qquad \vec{v}_i = (W_i,2\eps), \qquad \tilde{\vec{u}} = \frac{1}{p} \sum_{i=1}^n u_i\vec{v}_i = \ps{\vec{u}, \frac{2\sum_{i=1}^{n-1} \sigma_i + 1}{p}\eps}\]

for $2\leq i\leq n$.
%\[ \tilde{\bb v}_i = (\bb v_i, 2\eps),\qquad \tilde{\bb v}_n = (\bb v_n,\eps), \qquad \tilde{\bb u} = \frac{1}{p} \sum_{i=1}^n u_i\tilde{\bb v}_i = \ps{u, \frac{2\sum_{i=1}^{n-1} \sigma_i + 1}{p}\eps}.\]

We define $\tilde{\lat} \defeq \lat[\vec{v}_1,\ldots,\vec{v}_n,\tilde{\vec{u}}] \subset \R^{n+1}$ and $\mathcal{S} = \set{\vec{0}, \pm\vec{v}_1,\ldots,\pm\vec{v}_n,\pm \tilde{\vec{u}}}$.

%We define $\tilde{\lat} \defeq \lat[\tilde{\bb v}_1,\ldots,\tilde{\bb v}_n,\tilde{\bb u}] \subset \R^{n+1}$ and $\mathcal{S} = \set{\bb 0, \pm\tilde{\bb v}_1,\ldots,\pm\tilde{\bb v}_n,\pm \tilde{\bb u}}$.

\end{construction}

We first prove that the noise we added does not create any new short vectors in the lattice. However, it did alter the lengths of the shortest vectors slightly.

\begin{lemma} \label{prime lat}
$\tilde{\lat}$ is an $n$-dimensional lattice and moreover
\begin{itemize}
    \item $\tilde{\lat} \cap \B_q(R) = \mathcal{S}$---the only vectors in $\tilde{\lat}$ whose $\ell^q$ norm is at most $R$ is $\mathcal{S}$.
    \item  $\pm \vec{v}_1$ is the shortest vector in the lattice $\tilde{\lat}$.
\end{itemize}
\end{lemma}

\begin{proof}
We observe that $(\vec{v}_1,\ldots,\vec{v}_n)$ are linearly independent, while $\tilde{\vec{u}} \in \spn_\R(\vec{v}_1,\ldots,\vec{v}_n)$; thus, the dimension of the lattice is $n$.
% \todo{It seems like we are increasing the dimension by $1$ and decreasing the rank by $1$? although we can always multiply by the transpose of an orthonormal matrix...} \todo{Itamar: The rank is the same, the  dimension increase by 1}

\noindent
Let us calculate the length of the vectors in $\mathcal{S}$:

\begin{itemize}
\item $\norm{\vec{v}_1}_q=\ps{p^q+\eps^q}^\frac{1}{q}$.
\item $\norm{\vec{v}_i}_q = \ps{p^q+2^q\eps^q}^\frac{1}{q}$ for $2\le i \le n$.
\item 
\[ \norm{\tilde{\vec{u}}}_q =\ps{p^q + \frac{(2\sum_{i=1}^{n-1} \sigma_i + 1)^q}{p^q}\eps^q}^\frac{1}{q} \le \ps{p^q + \frac{((n-1)(p-1)+1)^q}{p^q}\eps^q }^\frac{1}{q} < R \]
\end{itemize}
Now, if $q>1$, then by \Cref{remark:n and p}, $2p<4n-3$,
\[ \norm{\tilde{\vec{u}}}_q \geq \ps{p^q + \frac{(4n-3)^q}{p^q}\eps^q}^\frac{1}{q} > \ps{p^q+2^q\eps^q}^\frac{1}{q}, \]
and thus $\norm{\vec{v}_1}<\norm{\vec{v}_i}_q<\norm{\tilde{\vec{u}}}_q<R$ for any $2\leq i \leq n$. If $q=1$, then
\[ \norm{\tilde{\vec{u}}}_1=p+\frac{2(p-1)+1}{p}\epsilon < p+2\epsilon = \norm{\vec{v}_n}_1 < R \]
and 
\[ \norm{\tilde{\vec{u}}}_1=p+(2-\frac{1}{p})\epsilon > p +\epsilon = \norm{\vec{v}_1}_1. \]

We now show that $\mathcal{S}$ contains the shortest vectors in $\tilde{\lat}$, i.e., any other lattice vector has norm larger than $R$. Let $\vec{x}=(x_1,\dots,x_{n+1})\in\tilde{\lat}\setminus\mathcal{S}$. Then, the vector $\vec{x}'=(x_1,\dots,x_n)\in\lat_+$ has norm strictly greater than $R$ by definition and the result follows.

%Indeed, for any $\bb w=(w_1,\ldots,w_{n+1})$ in $ \lat \setminus \mathcal{S}$, the vector $\bb w'=(w_1,\ldots,w_n)$ is a vector in $\lat_+$ that is not in $\mathcal{S}$, and hence its length is at least $R$ (by definition of $R$).
\end{proof}

\begin{lemma}
Any short basis of the lattice $\tilde{\lat}$ (where $\max_{1\le i\le n} \norm{M_i}_q$ is minimized) cannot include the shortest non-zero vector of $\tilde{\lat}$.
\end{lemma}
This Lemma implies \Cref{thm:Short basis vs. shortest vector}.

\begin{proof}
The lattice $\tilde{\lat}$ has a basis $\vec{v}_2,\ldots,\vec{v}_{n},\tilde{\vec{u}}$ (as $\vec{v}_1 = p\vec{\tilde u}-\sum_{i=2}^n \sigma_i \vec{v}_i$)
% \todo{$\tilde{\vec{u}}$? otherwise dimensions don't match}, 
hence the length of its shortest basis is at most $\norm{\tilde{\vec{u}}}$.
By \Cref{prime lat}, any basis of the lattice $\tilde{\lat}$ of length at most $\norm{\tilde{\vec{u}}}$, may  include only vectors in $\mathcal{S}$.
Such a basis must include $\tilde{\vec{u}}$, since the lattice spanned by $\vec{v}_1,\ldots,\vec{v}_n$ only has integer multiples of $p$ in its first coordinates, and hence $\lat[\vec{v}_1,\ldots,\vec{v}_n] \subsetneq \tilde{\lat}$.
Therefore, (up to change of signs) the basis must be $M^{(i)}=\set{\vec{v}_k \mid k \neq i} \cup \set{\tilde{\vec{u}}}$ for some $1\le i \le n$. \\
If $i \neq 1$, we claim that $\lat[M^{(i)}] \neq \tilde{\lat}$. 
Indeed the $i$-th coordinate of all the vectors in $\lat[M^{(i)}]$ will be in $\sigma_i\Z$, unlike our lattice, which has a vector with $p$ in its $i$-th coordinate, and of course $p \notin \sigma_i\Z$.  \\
So, it must be the case that $i=1$, that is, the shortest basis does not include $\pm \vec{v}_1$, the shortest lattice vector. %\todo{it remains to show that $\vec{v}_n\in\lat[\vec{v}_1,\dots,\vec{v}_{n-1},\tilde{\vec{u}}]$. Also it will probably be cleaner to have $\vec{v}_1$ be the shortest vector - i can change this soon}
\end{proof}

\begin{corollary}
There exists a lattice $\lat \subset \R^n$ such that any basis of $\lat$ of minimal $\ell^{q'}$ length cannot include the shortest lattice vector.
\end{corollary}

\begin{proof}
Any essentially distinct basis $M'$ of $\tilde{\lat}$ (not only up to signs/order of basis vectors) than 
$M = [\vec{v}_2,\ldots,\vec{v}_n,\tilde{\vec{u}}]$, must include a vector longer than $2^\frac{1}{q}p$, and hence the basis length is at least
\[\norm{M'}_{q'} \geq {\ps{\ps{2^\frac{1}{q}p}^{q'} + (n-1)\cdot p^{q'}}}^\frac{1}{q'} > pn^\frac{1}{q'} \]
while 
\[ \norm{M}_q \leq \ps{p^q + \max \ps{\frac{\ps{(n-1)(p-1)+1}^q}{p^q},2^q}\eps^q } \cdot n^\frac{1}{q'}. \]
    Therefore, when $\eps$ decreases to zero, $\lim_{\eps \rightarrow 0} \norm{M}_q = pn^\frac{1}{q'} < \norm{M'}_q$.
    So, one can choose small enough $\eps$ such that $\norm{M} < \norm{M'}_q$ (observe that $\eps$ does not depend in $M'$, only in $p,n,\eps,q$).
    In particular, any basis of $\lat$ that include $\vec{v}_n$ (the shortest vector of $\tilde{\lat}$) is longer in "$\ell^{q'}$ sense'' then $M$.
\end{proof}

% \begin{remark}
%     The same statement is correct for more general norms on $M$ such that $\norm{M}=N(\norm{M_1}_2,\ldots,\norm{M_k}_2)$
%     where $N:\R^k\rightarrow \left[0,\infty\right)$ is a norm function that is strictly monotonic in each of the coordinates. 
% \end{remark}

% \section{Further Discussion} \label{section:Discussion}

% The argument we have presented can be easily generalized to any  $\ell^{q}$-norm, as long as $1\leq q < \infty$.
% It is not clear what is true in the case $\ell^\infty$.

% \begin{conj}
%     Any lattice $\lat \subset \R^n$ is standard with respect to $\ell_\infty$ --- It has a basis $M$ such that $\norm{M_i}_\infty  = \lambda_i^{(\infty)}(\lat)$.
% \end{conj}

% Moreover, in \Cref{example:non-standard}, we have seen that with respect of the $\ell^{q}$ norm, there exists a $n>2^q$ dimensional lattice that is not standard. 
% We conjecture that this is tight: 
% \begin{conj}
% Any lattice of dimension $n\leq 2^q$ is standard with respect to the $\ell^q$ norm.
% \end{conj}
% This statement is true for $q\leq 2$ (it has been proved in \cite{feng2017standard} for $q=2$, and the proof is easily generalized using the relation between $\ell^q$ norms). 

\section{Acknowledgment}

We heartily thank Shvo Regavim for insightful discussions as well as important remarks on earlier draft of this paper.

\newpage
\appendix

\section{A Short Vector Modulo Prime}
\label{sec:A short vector modulo prime}

\begin{claim} \label{claim:euclid}
There exist $N_0 \in \N$ such that for any integer $n \geq N_0$, there exists integers $n_1,n_2 \geq 2\cdot 3^q$ such that $n=n_12^q + n_23^q$.
\end{claim}

\begin{proof}
As $\mathrm{gcd}(2^q,3^q)=1$ there exists $a,b \in \Z$ such that $a2^q+b3^q=1$, and without loss of generality $a>0,b<0$. Observe that
\[ 1 =a2^q+b3^q = (a-d3^q)2^q + (b+d2^q)3^q. \]
One can take large enough $d\in \N$ such that $b'=b+d2^q>0$.
Now, denote $N_0=((2+d)3^q)2^q+(2\cdot 3^q-b)3^q$.
We prove by induction that any $n \geq N_0$ can be written as $n=n_12^q+n_23^q,n_1,n_2 \geq 3^q$ (the induction basis is induced by that $(2+d)3^q,2\cdot 3^q-b \geq 2\cdot 3^q$).
Assume $n=n_1 2^q+n_2 3^q$, $n_1,n_2 \geq 3^q$ and therefore
\[ n+1 = (n_1+a)2^q+(n_2+b)3^q = (n_1+a-d3^q)2^q + (n_2+b+d2^q)3^q. \]
If both $n_2+b \leq 2\cdot 3^q$ and $n_1+a-d\cdot 3^q \leq 2 \cdot 3^q$, then $n_1 \leq (2+d)3^q-a,n_2 \leq 2\cdot 3^q-b$ and so 
\[n_12^q+n_23^q \leq ((2+d)3^q-a)2^q + (3^q-b)3^q<(2+d)3^q\cdot 2^q + (2\cdot 3^q-b)\cdot 3^q=N_0,\]
which is a contradiction.
Therefore, it must be the case that $n_1+a,n_2+b>2\cdot 3^q$, or $n+1+a-d3^q,n+2+b+d2^q>2\cdot 3^q$.
\end{proof}

\begin{claim}
[Existence of A Short Vector]
For any large enough prime, there exists $n\in \N$, $p<n-1$ and  $\mbox{\boldmath $\sigma$} \in \set{2,3}^{n-1}$ such that $\norm{\mbox{\boldmath $\sigma$} \mod p}_q = \ps{p^q-1}^\frac{1}{q}$, and for any $k \in \Z_p \setminus \set{0,1,-1}$, $\norm{k\mbox{\boldmath $\sigma$} \mod p} > \norm{\mbox{\boldmath $\sigma$} \mod p}$.    
\end{claim}

\begin{proof}
We take $N_0$ as in \Cref{claim:euclid}, and as long as $p^q-1\geq N_0$, there are $n_1,n_2$ such that $p^q-1 = 2^qn_1 + 3^qn_2$, and define $\mbox{\boldmath $\sigma$} =(2,\ldots,2,3\ldots,3)$, where the run of $2$'s is of length $n_1$, and the run of $3$'s is of length $n_2$.
Assume by way of contradiction that $\norm{k\mbox{\boldmath $\sigma$}\mod p}_q \le \norm{\mbox{\boldmath $\sigma$} \mod p}_q$. Denote $m_1 = \abs{2k}_p,m_2 = \abs{3k}_p$.
\[ 2^qn_1 + 3^qn_2 = p^q-1 = {(\norm{\mbox{\boldmath $\sigma$} \mod p}_q)}^q \geq (\norm{k \cdot \mbox{\boldmath $\sigma$} \mod p}_q)^{q} = n_1m_1^q + n_2m_2^q,\]
so $m_1 \le 2$ or $m_2 \le 3$.  
Since $m_1=2$ if and only if $k\equiv \pm 1 \mod p$ and $m_2=3$ if and only if $k \equiv \pm 1 \mod p$, we conclude that $m_1 < 2$ or $m_2 <3$. So, $m_1 = 1$ or $m_2 =  1$ or $m_2 =  2$, and we examine each one of the cases:

\begin{itemize}
\item $m_1 =  1$, hence $2k = \pm 1$ and $k = \frac{p\pm 1}{2}$, so $3k \equiv \frac{3p \pm 3}{2} \equiv \pm \frac{p + 3}{2}$, and $\abs{3k}_p = m_2 = \frac{p-3}{2}$.
\item $m_2 =  1$, hence $3k = \pm 1$, and
\begin{itemize}
\item If $p \equiv 1 \mod 3$, $k = \pm \frac{2p+1}{3}$, and $2k = \pm \frac{4p+2}{3} \equiv \pm \frac{p+2}{3}$, and $m_1 = \abs{2k}_p = \frac{p+2}{3}$.
\item If $p \equiv 2 \mod 3$, $k = \pm \frac{p+1}{3}$, and $2k = \pm \frac{2p+2}{3} \equiv \mp \frac{p-2}{3}$.
\end{itemize}
\item $m_2 =  2$, hence $3k = \pm 2$, and
\begin{itemize}
\item If $p \equiv 1 \mod 3$, $k = \pm \frac{p+2}{3}$, and $2k = \pm \frac{2p+4}{3} \equiv \mp \frac{p-4}{3}$.
\item If $p \equiv 2 \mod 3$, $k = \pm \frac{2p+2}{3}$, and $2k = \pm \frac{4p+4}{3} \equiv \mp \frac{p+4}{3}$.
\end{itemize}
\end{itemize}
In particular in each of the cases
$ p^q > p^q -1 \geq n_1m_1^q + n_2m_2^q \geq 2\cdot 3^q \cdot (\frac{p-4}{3})^q=2(p-4)^q$
which is a contradiction for large enough $p$.

\end{proof}

\begin{remark} \label{remark:n and p}
In this construction, $n = 1 + n_1 + n_2$. Since $2^qn_1+3^qn_2 = p^q-1$, $ n \geq 1 + \frac{p^q-1}{3^q}>\frac{p}{2}+\frac{3}{4}$
for large enough $p$ as long as $q>1$, and so if $q>1$, $4n-3>2p$.
\end{remark}

Moreover, for the $q=2$ case, we are looking for a prime $p$ such that $p^2-1=4n_1+9n_2$ and $n_1,n_2 \geq 18$. As $17^2-1=288=13\cdot 4 + 19 \cdot 9$, $n=13+19+1=34$ is the dimension of the counter-example.
Actually, a smaller "short vector modulo prime" exists, for $p=13$. The vector
\[  \mbox{\boldmath $\sigma$} = (2, 2, 2, 2, 2, 2, 2, 2, 2, 2, 2, 2, 2, 2, 2, 3, 3, 3, 3, 3, 3, 3, 3, 3, 3, 3, 3) = (\bb 2^{15},\bb 3^{12}) \]
and hence the dimension of the counter-example lattice is $n=28$.
Ignoring the 2-3 limitation, we used a random algorithm for finding a fit and short as possible vector. The vector $(\bb 2^3,\bb 3^3,\bb 4^1,\bb 5^4,\bb 6^3,\bb 7^2)$ over $p=19$ is such a vector, and gives a construction of dimension $18$. Proving this fact will remain as an exercise to the reader.
\newpage

\bibliography{references}

\end{document}

%% file: intro.tex
Broadly, a lattice is defined as a discrete subgroup of a topological group.\footnote{Moreover, an invariant and finite measure of the quotient group must exist.}
Here, we are interested in discrete subgroups of $\R^n$. 
Such a subgroup $\lat$ can be described as the integer linear combination of a collection of linear independent vectors $\vec{v}_1,\ldots,\vec{v}_k \in \R^n$: 
\[\lat={\spn}_{\Z}(\vec{v}_1,\ldots,\vec{v}_k)\defeq\sett{\sum_{i=1}^k n_i \vec{v}_i}{n_1,\ldots,n_k \in \Z}.\]
The collection $\{\vec{v}_i\}$ is a \emph{basis} of the lattice $\lat$. 

Historically, lattices have been investigated since the late 18th century by great mathematicians including Lagrange, Gauss, and Minkowski~\cite{minkowski1910geometrie} (mostly for number theoretic applications).

In the late 20th century, lattices became a hot topic of transformative  research with regards to the complexity of solving related computational problems.
Research in the area has thrived in several directions. One of those is the quest for algorithms to solve computational problems over lattices as efficiently as possible. 

The fundamental mathematical principles involved are crucial in various distinct fields.
The theorems proven by Minkowski~\cite{minkowski1896geometrie}, more than a century ago, and the research program he laid out, continue to be very relevant.
Some aspects affect the complexity of computational problems over lattices, in particular, their average-case complexity, with natural applications to cryptography.

% Getting a better understand about the mathematics properties of lattices can help us in our research about those problems hardness and usages.

Note that a lattice may be represented by numerous bases.
% Some properties are abstract and true regardless of the choice of basis while others depend on the basis.
From a computational perspective,  a lattice is described by a particular basis and the complexity of the computational problem depends fundamentally on the specific basis given as input.
If the given basis is close to orthogonal, problem such as SVP could be computed efficiently, while, given an arbitrary basis, the problem may turn out to be infeasible to solve efficiently.
% \todo{The fundamental parallelepiped form by a basis might be very long in some directions while short in others---in high dimensions, this may be very jard to correct and the vector closest to a given vector might turn out in a completely different parallelepiped.}

% The overall theme is that, since a lattice is represented by a basis, the hardness of a problem fundamentally depends on which basis is given as input;
% SVP can be efficiently computed provided a basis close to orthogonal, in some sense;
% while, if the basis forms a long and narrow parallelepiped, the same problem could be extremely hard to solve.
% \todo{rephrase: The underlying idea is to use the fact that the same lattice can be represented by different bases, as such, the hardness of a problem fundamentally depends on which basis is given as input. If the input basis is close to orthogonal, then SVP can be efficiently computed;
% however, if the basis is 'bad' (e.g., it forms a long and narrow parallelepiped), the same SVP problem could be extremely hard to solve.}

In other words, the hardness of some computational problems over lattices stems from the hardness of finding a `good' basis given an arbitrary one.
The problem of finding the best lattice basis is well-studied (see~\cite{lll, micciancio2002complexity}), as well as finding a short lattice vector/a collection of short lattice vectors (see~\cite{aks, Ajtai1, pohst1981computation} just for starters).

\subsection{Best Lattice Basis}
There are few distinct (not necessarily equivalent) plausible definitions for the \emph{best basis} for a lattice;
these formulate different manners by which to assert that the basis is either \emph{short} or close as possible to being \emph{orthogonal}.
Several options have been proposed for how to define what it means for a basis to be short.
For instance, it could be the length of the longest basis vector, the sum of lengths, or the product of lengths.

Consider the following \emph{greedy} algorithm (not necessarily efficient) to find such a basis:
start with the shortest vector in the lattice, and then add the next shortest vector, albeit, one that is linearly independent of previous basis vectors (that is, not in the span of the current partial basis), etc.
Those vectors achieve, in terms of norms, the \emph{successive minima} (see \Cref{definition: lambdas}). 
Nevertheless, the successive minima do not necessarily form a basis for the lattice;
in fact, there are canonical examples of lattices such as \Cref{example:non-standard} that have no basis that achieves the successive minima: in these cases, the successive minima spans only a sub-lattice.

In this paper, we completely characterize \emph{non-standard} lattices; 
that is, lattices that have no basis which achieves the successive minima. Informally, we prove that any \emph{non-standard lattice} is a "faulty'' \emph{standard lattice}. Such characterization could significantly improve our understanding of non-standard lattices.
Moreover, they may facilitate improved techniques regarding computability issues.

Let $\lat$ be a non-standard lattice.
Take any sub-lattice $\lat_{\mathrm{std}}\subset \lat$ that has a basis whose vectors' length agree with the successive minima of $\lat$---the lattice spanned by these vectors is a standard sub-lattice of $\lat$.
Observe that it is maximal, in the sense that any other \emph{sub-standard} lattice may not have smaller successive minima.

Then, the lattice $\lat$ is a product of the standard lattice $\lat_{\mathrm{std}}$ with a rational matrix.
Specifically, select a subset  of the vectors in $\lat_{\mathrm{std}}$, divide them each by an arbitrary natural number and add the scaled vectors to the lattice $\lat_{\mathrm{std}}$. $\lat$ will be the additive closure of $\lat_{\mathrm{std}}$ and the scaled vectors (namely, the minimal lattice containing all):
see \Cref{theorem: Characterization of non-standard lattices}.

\subsection{Best Lattice Basis: Revisited}

Once it has been established that the most natural greedy process (which collects vectors that achieve the successive minima) might fail, one considers other greedy processes for finding a short basis.

Here is an even more fundamental question: does every lattice have a basis that includes the shortest vector?

The answer is clearly: `Yes'---the Korkine–Zolotarev (KZ) basis is one that includes the shortest lattice vector \cite{kz}. 
In general, any lattice vector that is not an integer multiple of another lattice vector,
%cannot be divided by a natural number greater than one to obtain a vector in the lattice,
can be completed into a basis of the lattice.

Let the shortest basis be the lattice basis which minimizes the longest basis vector. Consider the following more complex question:
will the shortest basis always include the shortest vector?
(The KZ process produce a basis that might be $O(\sqrt{n})$-approximation of the shortest basis).

The answer, somewhat surprisingly, established in~\cite{regavim2021minkowski}, is `No'.

\begin{theorem}[\cite{regavim2021minkowski}] \label{thm:shvo}
    There is a lattice $\lat \subset \R^{43}$ such that no basis $M=\left[\vec{v}_1,\ldots,\vec{v}_{43}\right]$ of $\lat$ can satisfy the following two properties simultaneously: 
    \begin{itemize}
        \item $M$ includes a shortest lattice vector: there exists an $i$ such that $\norm{\vec{v}_i}_2 = \lambda_1(\lat)$.
        \item $M$ is the shortest basis of $\lat$, that is, $M$ minimizes the longest basis vector among all bases of $\lat$. %the longest vector in $M$ is the shortest between all the bases of $\lat$.
    \end{itemize}
\end{theorem}

We extend this theme and exhibit a family of constructive lattices for which the shortest basis never includes the shortest vector.
Interestingly, the lattices we present show the same statement for general norms, namely, any $\ell^q$ norm for any $q\in\N$.

We further extend this statement and prove that the same holds when the basis is minimal with respect to the \emph{sum} of the basis vectors' lengths;
it still applies to many other ``reasonable" manners to measure the basis length.

\begin{theorem} \label{thm:main}
    Let $p,q \in \N$.
    There is a lattice $\lat \subset \R^n$ such that no basis $M=\left[ \vec{v}_1,\ldots,\vec{v}_n \right]$ of $\lat$ satisfies the following two properties simultaneously: 
    \begin{itemize}
        \item $M$ includes a shortest lattice vector: $\norm{\vec{v}_i}_q = \lambda_1^{(q)}(\lat)$.
        \item $M$ is the shortest basis of $\lat$, namely, it minimizes the following function over  all bases of $\lat$:
        \[ \sum_{i=1}^n \norm{\vec{v}_i}_q^p \]
        or, if $p=\infty$, 
        \[ \max_{1\leq i\leq n} \norm{\vec{v}_i}. \]
    \end{itemize}
\end{theorem}

In particular, let $q=2$, $p=\infty$.
Applying the method we develop in this paper, we have found the following construction of an $18$-dimensional lattice, which re-proves \Cref{thm:shvo}, albeit with a lattice of smaller dimension. 

{ \scriptsize
\[ \begin{bmatrix*}
19 & 0 & 0 & 0 & 0 & 0 & 0 & 0 & 0 & 0 & 0 & 0 & 0 & 0 & 0 & 0 & 6\\
0 & 19 & 0 & 0 & 0 & 0 & 0 & 0 & 0 & 0 & 0 & 0 & 0 & 0 & 0 & 0 & 5\\
0 & 0 & 19 & 0 & 0 & 0 & 0 & 0 & 0 & 0 & 0 & 0 & 0 & 0 & 0 & 0 & 2\\
0 & 0 & 0 & 19 & 0 & 0 & 0 & 0 & 0 & 0 & 0 & 0 & 0 & 0 & 0 & 0 & 4\\
0 & 0 & 0 & 0 & 19 & 0 & 0 & 0 & 0 & 0 & 0 & 0 & 0 & 0 & 0 & 0 & 7\\
0 & 0 & 0 & 0 & 0 & 19 & 0 & 0 & 0 & 0 & 0 & 0 & 0 & 0 & 0 & 0 & 6\\
0 & 0 & 0 & 0 & 0 & 0 & 19 & 0 & 0 & 0 & 0 & 0 & 0 & 0 & 0 & 0 & 3\\
0 & 0 & 0 & 0 & 0 & 0 & 0 & 19 & 0 & 0 & 0 & 0 & 0 & 0 & 0 & 0 & 3\\
0 & 0 & 0 & 0 & 0 & 0 & 0 & 0 & 19 & 0 & 0 & 0 & 0 & 0 & 0 & 0 & 5\\
0 & 0 & 0 & 0 & 0 & 0 & 0 & 0 & 0 & 19 & 0 & 0 & 0 & 0 & 0 & 0 & 6\\
0 & 0 & 0 & 0 & 0 & 0 & 0 & 0 & 0 & 0 & 19 & 0 & 0 & 0 & 0 & 0 & 3\\
0 & 0 & 0 & 0 & 0 & 0 & 0 & 0 & 0 & 0 & 0 & 19 & 0 & 0 & 0 & 0 & 5\\
0 & 0 & 0 & 0 & 0 & 0 & 0 & 0 & 0 & 0 & 0 & 0 & 19 & 0 & 0 & 0 & 5\\
0 & 0 & 0 & 0 & 0 & 0 & 0 & 0 & 0 & 0 & 0 & 0 & 0 & 19 & 0 & 0 & 2\\
0 & 0 & 0 & 0 & 0 & 0 & 0 & 0 & 0 & 0 & 0 & 0 & 0 & 0 & 19 & 0 & 2\\
0 & 0 & 0 & 0 & 0 & 0 & 0 & 0 & 0 & 0 & 0 & 0 & 0 & 0 & 0 & 19 & 7\\
0 & 0 & 0 & 0 & 0 & 0 & 0 & 0 & 0 & 0 & 0 & 0 & 0 & 0 & 0 & 0 & 1\\
\frac{1}{1000} & \frac{1}{1000} & \frac{1}{1000} & \frac{1}{1000} & \frac{1}{1000} & \frac{1}{1000} & \frac{1}{1000} & \frac{1}{1000} & \frac{1}{1000} & \frac{1}{1000} & \frac{1}{1000} & \frac{1}{1000} & \frac{1}{1000} & \frac{1}{1000} & \frac{1}{1000} & \frac{1}{2000} & \frac{179}{2000\cdot19}
\end{bmatrix*} \] }

These extensions are somewhat stronger than what is established in \cite{regavim2021minkowski}. We believe that our proof is simpler, albeit \emph{de gustibus non disputandum}.

\subsection{Organization}

In \Cref{sec:characterization} we present and prove a characterization theorem of standard and non standard lattices.
In \Cref{sec:construction} we prove \Cref{thm:main}, and describe a lattice with no \emph{``ultimately good"} basis. %\todo{this is the only place where we mention 'ultimately good'} 
\Cref{sec:construction} comprises of three parts:
in \Cref{subsec:plan} we describe the construction details on a high level;
in \Cref{subsec:stage1} we describe the first stage of the construction, in which we add a unit vector to the grid, without increasing its dimension, or making it too dense.
This construction utilizes a special vector, which is ``short modulo prime". 
The proof of the vector's existence is described in \Cref{sec:A short vector modulo prime}. In \Cref{subsec:fixing} we add some noise to the extended grid, and use the new lattice to complete the proof of \Cref{thm:main}. We discuss possible extensions of our results in \Cref{section:Discussion}.

%% file: draw.tex
\begin{figure}[!ht]
    \centering
    \tdplotsetmaincoords{60}{125}
    \tdplotsetrotatedcoords{0}{0}{0} %<- rotate around (z,y,z)
    \begin{subfigure}[ht]{0.4\textwidth}
        \centering
        \resizebox{\linewidth}{!}{
            \begin{tikzpicture}[scale=1.7,tdplot_rotated_coords,
                    cube/.style={very thick,black},
                    grid/.style={very thin,gray},
                    axis/.style={-stealth,darkgray, thick},
                    rotated axis/.style={->,purple,ultra thick}]

    %draw a grid in the x-y plane
    \foreach \x in {-0.5,0,...,2.5}
        \foreach \y in {-0.5,0,...,2.5}
        {
            \draw[grid] (\x,-0.5) -- (\x,2.5);
            \draw[grid] (-0.5,\y) -- (2.5,\y);
        }

    % %draw the main coordinate frame axes
    \draw[axis,tdplot_main_coords] (0,0,0) -- (3,0,0) node[anchor=west]{$\,x$};
    \draw[axis,tdplot_main_coords] (0,0,0) -- (0,2.9,0) node[anchor=north west]{$y$};
    \draw[axis,tdplot_main_coords] (0,0,0) -- (0,0,2.5) node[anchor=west]{$z$};

    % %draw the rotated coordinate frame axes
    % \draw[rotated axis] (0,0,0) -- (3,0,0) node[anchor=west]{$x'$};
    % \draw[rotated axis] (0,0,0) -- (0,3,0) node[anchor=south west]{$y'$};
    % \draw[rotated axis] (0,0,0) -- (0,0,3) node[anchor=west]{$z'$};

    %draw the top and bottom of the cube
    \draw[cube,fill=blue!5,opacity=0.7] (0,0,0) -- (0,2,0) -- (2,2,0) -- (2,0,0) -- cycle;

    %draw the top and bottom of the cube
    \draw[cube,fill=red!5,opacity=0.7] (0,0,0) -- (0,2,0) -- (0,2,2) -- (0,0,2) -- cycle;

    %draw the top and bottom of the cube
    \draw[cube,fill=green!5,opacity=0.7] (0,0,0) -- (2,0,0) -- (2,0,2) -- (0,0,2) -- cycle;

\foreach \x in {0,1,2}
   \foreach \y in {0,1,2}
      \foreach \z in {0,1,2}{
           %#####################################################
           \ifthenelse{  \lengthtest{\x pt < 2pt}  }
           {
             % True
                \draw [black]   (\x,\y,\z) -- (\x+1,\y,\z);
           }
           {% False
           }
           %#####################################################
           \ifthenelse{  \lengthtest{\y pt < 2pt}  }
           {
             % True
                \draw [black]   (\x,\y,\z) -- (\x,\y+1,\z);
           }
           {% False
           }
           %#####################################################
           \ifthenelse{  \lengthtest{\z pt < 2pt}  }
           {
             % True
                \draw [black]   (\x,\y,\z) -- (\x,\y,\z+1);
           }
           {% False
           }
           \shade[rotated axis,ball color = purple!80] (\x,\y,\z) circle (0.06cm);
}
% \foreach \x in {-1,0}
%    \foreach \y in {-1,0}
%       \foreach \z in {-1,0,1}{
%             \ifthenelse{  \lengthtest{\x pt < 0pt}  }
%             {
%             \draw [gray] [dashed]  (0.53+\x,0.8+\y,0.26+\z) -- (1.53+\x,0.8+\y,0.26+\z);
%             }
%             {}
%             \ifthenelse{  \lengthtest{\y pt < 0pt}  }
%             {
%             \draw [gray] [dashed]  (0.53+\x,0.8+\y,0.26+\z) -- (0.53+\x,1.8+\y,0.26+\z);
%             }
%             {}
%             \ifthenelse{  \lengthtest{\z pt < 1pt}  }
%             {
%             \draw [gray] [dashed]  (0.53+\x,0.8+\y,0.26+\z) -- (0.53+\x,0.8+\y,1.26+\z);
%             }
%             {}
            %\draw [gray] {\x,\y,\z} -- (1+\x,1+\y,1+\z);
            %\shade[rotated axis,ball color = green!80] (0.53+\x,0.8+\y,0.26+\z) circle (0.06cm) ;  
            %\draw [red][-stealth](\x+1,\y+1,\z+1) -- (0.53+\x,0.8+\y,0.26+\z);
            \node at (1.3,0.05,0.35) {$ W_1$};
            \node at (0.7,1.5,0.3) {$W_2$};
            \node at (0.9,0.4,1.65) {$ W_3$};
            %\node at (0.53,0.8,0.26) {$u$};
\end{tikzpicture}
         }
        \caption{We scale the integer lattice $\Z^n$ by a prime $p$.  }
        \label{fig:subfig8}
    \end{subfigure}
    \hspace*{\fill}
    \begin{subfigure}[ht]{0.4\textwidth}
    \centering
        \resizebox{\linewidth}{!}{
            \begin{tikzpicture}[scale=1.7,tdplot_rotated_coords,
                    cube/.style={very thick,black},
                    grid/.style={very thin,gray},
                    axis/.style={-stealth,darkgray, thick},
                    rotated axis/.style={->,purple,ultra thick}]

    %draw a grid in the x-y plane
    \foreach \x in {-0.5,0,...,2.5}
        \foreach \y in {-0.5,0,...,2.5}
        {
            \draw[grid] (\x,-0.5) -- (\x,2.5);
            \draw[grid] (-0.5,\y) -- (2.5,\y);
        }

    % %draw the main coordinate frame axes
    \draw[axis,tdplot_main_coords] (0,0,0) -- (3,0,0) node[anchor=west]{$\,x$};
    \draw[axis,tdplot_main_coords] (0,0,0) -- (0,2.9,0) node[anchor=north west]{$y$};
    \draw[axis,tdplot_main_coords] (0,0,0) -- (0,0,2.5) node[anchor=west]{$z$};

    % %draw the rotated coordinate frame axes
    % \draw[rotated axis] (0,0,0) -- (3,0,0) node[anchor=west]{$x'$};
    % \draw[rotated axis] (0,0,0) -- (0,3,0) node[anchor=south west]{$y'$};
    % \draw[rotated axis] (0,0,0) -- (0,0,3) node[anchor=west]{$z'$};

    %draw the top and bottom of the cube
    \draw[cube,fill=blue!5,opacity=0.7] (0,0,0) -- (0,2,0) -- (2,2,0) -- (2,0,0) -- cycle;

    %draw the top and bottom of the cube
    \draw[cube,fill=red!5,opacity=0.7] (0,0,0) -- (0,2,0) -- (0,2,2) -- (0,0,2) -- cycle;

    %draw the top and bottom of the cuaaaaaabe
    \draw[cube,fill=green!5,opacity=0.7] (0,0,0) -- (2,0,0) -- (2,0,2) -- (0,0,2) -- cycle;

\foreach \x in {0,1,2}
   \foreach \y in {0,1,2}
      \foreach \z in {0,1,2}{
           %#####################################################
           \ifthenelse{  \lengthtest{\x pt < 2pt}  }
           {
             % True
                \draw [black]   (\x,\y,\z) -- (\x+1,\y,\z);
           }
           {% False
           }
           %#####################################################
           \ifthenelse{  \lengthtest{\y pt < 2pt}  }
           {
             % True
                \draw [black]   (\x,\y,\z) -- (\x,\y+1,\z);
           }
           {% False
           }
           %#####################################################
           \ifthenelse{  \lengthtest{\z pt < 2pt}  }
           {
             % True
                \draw [black]   (\x,\y,\z) -- (\x,\y,\z+1);
           }
           {% False
           }
           \shade[rotated axis,ball color = purple!80] (\x,\y,\z) circle (0.06cm);  
}
\foreach \x in {-1,0}
   \foreach \y in {-1,0}
      \foreach \z in {-1,0,1}{
            \ifthenelse{  \lengthtest{\x pt < 0pt}  }
            {
            \draw [gray] [dashed]  (0.53+\x,0.8+\y,0.26+\z) -- (1.53+\x,0.8+\y,0.26+\z);
            }
            {}
            \ifthenelse{  \lengthtest{\y pt < 0pt}  }
            {
            \draw [gray] [dashed]  (0.53+\x,0.8+\y,0.26+\z) -- (0.53+\x,1.8+\y,0.26+\z);
            }
            {}
            \ifthenelse{  \lengthtest{\z pt < 1pt}  }
            {
            \draw [gray] [dashed]  (0.53+\x,0.8+\y,0.26+\z) -- (0.53+\x,0.8+\y,1.26+\z);
            }
            {}
            %\draw [gray] {\x,\y,\z} -- (1+\x,1+\y,1+\z);
            %\draw [orange][->](\x+1,\y+1,\z+1) -- (0.53+\x,0.8+\y,0.26+\z);
            \shade[rotated axis,ball color = green!80] (0.53+\x,0.8+\y,0.26+\z) circle (0.06cm) ;  
            \node at (1.3,0.05,0.35) {$W_1$};
            \node at (0.7,1.5,0.3) {$ W_2$};
            \node at (0.9,0.4,1.65) {$ W_3$};
            \node at (0.4,0.8,0.4) {$\bb u$};
}
\end{tikzpicture}
        }
        \caption{We add a certain vector of the form $\vec{u} = (1,2,...,2,3,...,3)$,
        %\todo{specifically this form, $2$s then $3$s then $1$? or any combo of $2$,$3$?}
        of length $p$ to the scaled grid $p\Z^n$, and get the lattice $\lat_+$. We choose a vector $\vec{u}$ such that all the (non-zero) vectors in the lattice will be of length at least $p$.}   
        \label{fig:subfig9}
    \end{subfigure}

    \begin{subfigure}[b]{0.6\textwidth}
        \centering
        \resizebox{\linewidth}{!}{
            \begin{tikzpicture}[scale=1.7,tdplot_rotated_coords,
                    cube/.style={very thick,black},
                    grid/.style={very thin,gray},
                    axis/.style={-stealth,darkgray, thick},
                    rotated axis/.style={->,purple,ultra thick}]

    %draw a grid in the x-y plane
    \foreach \x in {-0.5,0,...,2.5}
        \foreach \y in {-0.5,0,...,2.5}
        {
            \draw[grid] (\x,-0.5) -- (\x,2.5);
            \draw[grid] (-0.5,\y) -- (2.5,\y);
        }

    % %draw the main coordinate frame axes
    \draw[axis,tdplot_main_coords] (0,0,0) -- (3,0,0) node[anchor=west]{$\,x$};
    \draw[axis,tdplot_main_coords] (0,0,0) -- (0,2.9,0) node[anchor=north west]{$y$};
    \draw[axis,tdplot_main_coords] (0,0,0) -- (0,0,2.5) node[anchor=west]{$z$};

    % %draw the rotated coordinate frame axes
    % \draw[rotated axis] (0,0,0) -- (3,0,0) node[anchor=west]{$x'$};
    % \draw[rotated axis] (0,0,0) -- (0,3,0) node[anchor=south west]{$y'$};
    % \draw[rotated axis] (0,0,0) -- (0,0,3) node[anchor=west]{$z'$};

    %draw the top and bottom of the cube
    \draw[cube,fill=blue!5,opacity=0.5,draw opacity=0.34] (0,0,0) -- (0,2,0) -- (2,2,0) -- (2,0,0) -- cycle;

    %draw the top and bottom of the cube
    \draw[cube,fill=red!5,opacity=0.5,draw opacity=0.34] (0,0,0) -- (0,2,0) -- (0,2,2) -- (0,0,2) -- cycle;

    %draw the top and bottom of the cube
    \draw[cube,fill=green!5,opacity=0.5,draw opacity=0.34] (0,0,0) -- (2,0,0) -- (2,0,2) -- (0,0,2) -- cycle;

\foreach \x in {0,1,2}
   \foreach \y in {0,1,2}
      \foreach \z in {0,1,2}{
           %#####################################################
           \ifthenelse{  \lengthtest{\x pt < 2pt}  }
           {
             % True
                \draw [black]   (1.03*\x+0.02*\y,1.02*\y+0.03*\z,0.96*\z+0.01*\x+0.1*\y) -- (1.03*\x+1.03+0.02*\y,1.02*\y+0.03*\z,0.96*\z+0.01*\x+0.1*\y+0.01);
           }
           {% False
           }
           %#####################################################
           \ifthenelse{  \lengthtest{\y pt < 2pt}  }
           {
             % True
                \draw [black]   (1.03*\x+0.02*\y,1.02*\y+0.03*\z,0.96*\z+0.01*\x+0.1*\y) -- (1.03*\x+0.02*\y,1.02*\y+1.02+0.03*\z,0.96*\z+0.01*\x+0.1*\y+0.1);
           }
           {% False
           }
           %#####################################################
           \ifthenelse{  \lengthtest{\z pt < 2pt}  }
           {
             % True
                \draw [black]   (1.03*\x+0.02*\y,1.02*\y+0.03*\z,0.96*\z+0.01*\x+0.1*\y) -- (1.03*\x+0.02*\y,1.02*\y+0.03*\z+0.03,0.96*\z+0.96+0.01*\x+0.1*\y);
           }
           {% False
           }
           \shade[rotated axis,ball color = purple!80] (1.03*\x+0.02*\y,1.02*\y+0.03*\z,0.96*\z+0.01*\x+0.1*\y) circle (0.06cm);  

           \shade[rotated axis,ball color = purple!80,opacity=0.2] (\x,\y,\z) circle (0.06cm);  
}
\foreach \x in {-1,0}
   \foreach \y in {-1,0}
      \foreach \z in {-1,0,1}{
            %\draw [gray] {\x,\y,\z} -- (1+\x,1+\y,1+\z);
            \shade[rotated axis,ball color = purple!80] (0.5485+\x,0.8526+\y,0.2736+\z) circle (0.06cm) ;  
            \shade[rotated axis,ball color = purple!80,opacity=0.3] (0.53+\x,0.8+\y,0.26+\z) circle (0.06cm) ;  
            %\draw [red][-stealth](\x+1,\y+1,\z+1) -- (0.53+\x,0.8+\y,0.26+\z);
            \node at (1.3,0.1,0.35) {$\tilde{W}_1$};
            \node at (0.7,1.5,0.45) {$\tilde{W}_2$};
            \node at (0.9,0.4,1.55) {$\tilde{W}_3$};
            \node at (0.53,0.82,0.1) {$\tilde{\bb u}$};
}
\end{tikzpicture}
        }
        \caption{We add "linear" noise to the lattice vectors and get a new vector $\tilde{\lat}$. We add another coordinate to the basis vectors and put little values in them. We add a small enough noise such that there won't be any new small vectors in the lattice.}
        \label{fig:subfig10}
    \end{subfigure}
\caption{A Sketch of the Construction} 
\label{fig:subfig1.a.4}
\end{figure}
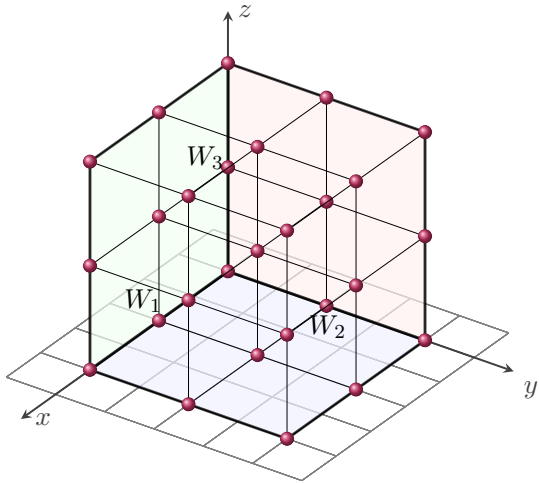
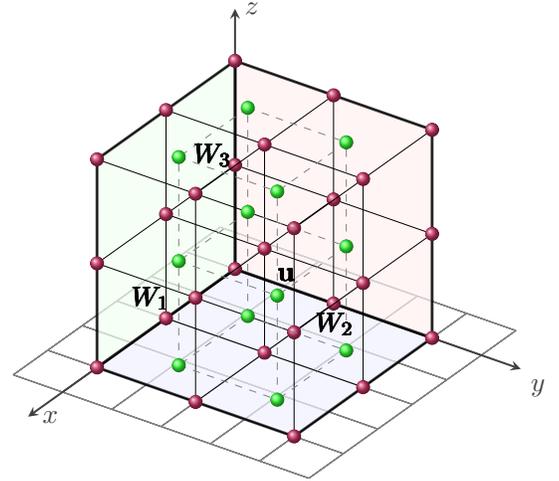
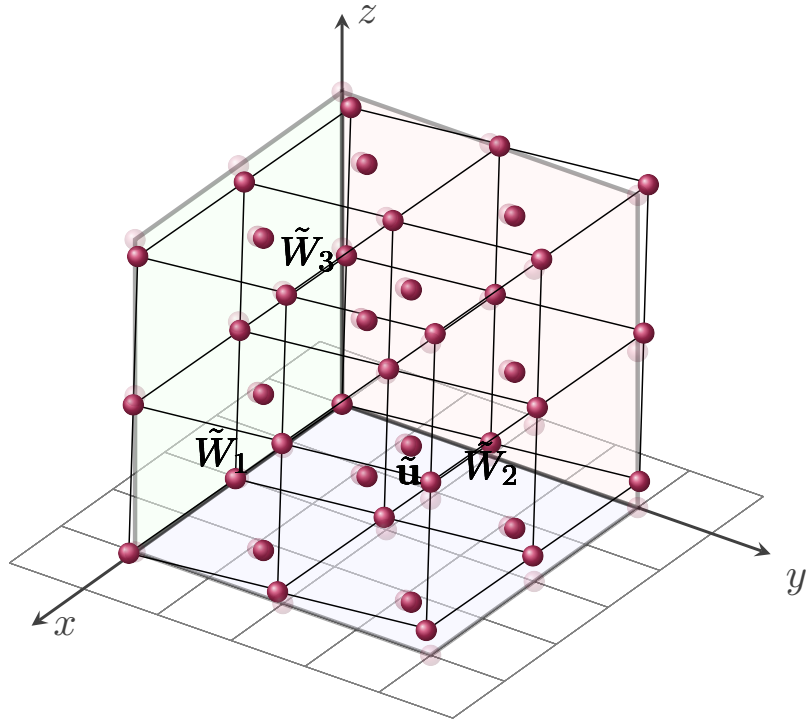